\documentclass[12pt]{amsart}
\usepackage{amssymb}
\usepackage{amsxtra}
\usepackage{graphics}
\usepackage{latexsym}
\usepackage{enumerate}
\usepackage{xcolor}
\usepackage{amsmath}
\usepackage{amssymb,amsthm,amsfonts}
\usepackage{mathtools}
\usepackage{amscd}
\usepackage[arrow, matrix, curve]{xy}
\usepackage[mathscr]{euscript}
\usepackage{syntonly}

\title[Linear Hodge-Newton decomposition and its applications]{Linear Hodge-Newton decomposition and its applications}

\author[Ziyan Song]{Ziyan Song}
\email{ziyans@mail.ustc.edu.cn}
 
\address{School of Mathematical Sciences,
University of Science and Technology of China, Hefei, 230026, China}

\begin{document}
\theoremstyle{plain}
\newtheorem{thm}{Theorem}[section]
\newtheorem{theorem}[thm]{Theorem}
\newtheorem*{theorem*}{Theorem}
\newtheorem*{definition*}{Definition}
\newtheorem{lemma}[thm]{Lemma}
\newtheorem{sublemma}[thm]{Sublemma}
\newtheorem{corollary}[thm]{Corollary}
\newtheorem*{corollary*}{Corollary}
\newtheorem{proposition}[thm]{Proposition}
\newtheorem{addendum}[thm]{Addendum}
\newtheorem{variant}[thm]{Variant}
\theoremstyle{definition}
\newtheorem{construction}[thm]{Construction}
\newtheorem{notations}[thm]{Notations}
\newtheorem{question}[thm]{Question}
\newtheorem{problem}[thm]{Problem}
\newtheorem{remark}[thm]{Remark}
\newtheorem{remarks}[thm]{Remarks}
\newtheorem{definition}[thm]{Definition}
\newtheorem{claim}[thm]{Claim}
\newtheorem{assumption}[thm]{Assumption}
\newtheorem{assumptions}[thm]{Assumptions}
\newtheorem{properties}[thm]{Properties}
\newtheorem{example}[thm]{Example}
\newtheorem{conjecture}[thm]{Conjecture}
\numberwithin{equation}{thm}

\newcommand{\sA}{{\mathcal A}}
\newcommand{\sB}{{\mathcal B}}
\newcommand{\sC}{{\mathcal C}}
\newcommand{\sD}{{\mathcal D}}
\newcommand{\sE}{{\mathcal E}}
\newcommand{\sF}{{\mathcal F}}
\newcommand{\sG}{{\mathcal G}}
\newcommand{\sH}{{\mathcal H}}
\newcommand{\sI}{{\mathcal I}}
\newcommand{\sJ}{{\mathcal J}}
\newcommand{\sK}{{\mathcal K}}
\newcommand{\sL}{{\mathcal L}}
\newcommand{\sM}{{\mathcal M}}
\newcommand{\sN}{{\mathcal N}}
\newcommand{\sO}{{\mathcal O}}
\newcommand{\sP}{{\mathcal P}}
\newcommand{\sQ}{{\mathcal Q}}
\newcommand{\sR}{{\mathcal R}}
\newcommand{\sS}{{\mathcal S}}
\newcommand{\sT}{{\mathcal T}}
\newcommand{\sU}{{\mathcal U}}
\newcommand{\sV}{{\mathcal V}}
\newcommand{\sW}{{\mathcal W}}
\newcommand{\sX}{{\mathcal X}}
\newcommand{\sY}{{\mathcal Y}}
\newcommand{\sZ}{{\mathcal Z}}
\newcommand{\A}{{\mathbb A}}
\newcommand{\B}{{\mathbb B}}
\newcommand{\C}{{\mathbb C}}
\newcommand{\D}{{\mathbb D}}
\newcommand{\E}{{\mathbb E}}
\newcommand{\F}{{\mathbb F}}
\newcommand{\G}{{\mathbb G}}
\newcommand{\HH}{{\mathbb H}}
\newcommand{\I}{{\mathbb I}}
\newcommand{\J}{{\mathbb J}}
\renewcommand{\L}{{\mathbb L}}
\newcommand{\M}{{\mathbb M}}
\newcommand{\N}{{\mathbb N}}
\renewcommand{\P}{{\mathbb P}}
\newcommand{\Q}{{\mathbb Q}}
\newcommand{\R}{{\mathbb R}}
\newcommand{\SSS}{{\mathbb S}}
\newcommand{\T}{{\mathbb T}}
\newcommand{\U}{{\mathbb U}}
\newcommand{\V}{{\mathbb V}}
\newcommand{\W}{{\mathbb W}}
\newcommand{\X}{{\mathbb X}}
\newcommand{\Y}{{\mathbb Y}}
\newcommand{\Z}{{\mathbb Z}}
\newcommand{\id}{{\rm id}}
\newcommand{\rank}{{\rm rank}}
\newcommand{\END}{{\mathbb E}{\rm nd}}
\newcommand{\End}{{\rm End}}
\newcommand{\Hom}{{\rm Hom}}
\newcommand{\Hg}{{\rm Hg}}
\newcommand{\tr}{{\rm tr}}
\newcommand{\Sl}{{\rm Sl}}
\newcommand{\Gl}{{\rm Gl}}
\newcommand{\Cor}{{\rm Cor}}
\newcommand{\Aut}{\mathrm{Aut}}
\newcommand{\Sym}{\mathrm{Sym}}
\newcommand{\ModuliCY}{\mathfrak{M}_{CY}}
\newcommand{\HyperCY}{\mathfrak{H}_{CY}}
\newcommand{\ModuliAR}{\mathfrak{M}_{AR}}
\newcommand{\Modulione}{\mathfrak{M}_{1,n+3}}
\newcommand{\Modulin}{\mathfrak{M}_{n,n+3}}
\newcommand{\Gal}{\mathrm{Gal}}
\newcommand{\Spec}{\mathrm{Spec}}
\newcommand{\res}{\mathrm{res}}
\newcommand{\coker}{\mathrm{coker}}
\newcommand{\Jac}{\mathrm{Jac}}
\newcommand{\HIG}{\mathrm{HIG}}
\newcommand{\MIC}{\mathrm{MIC}}

\maketitle

\begin{abstract}
Firstly, we provide a different proof of an important lemma in Buzzard and Calegari's work on slopes of overconvergent 2-adic modular forms via nonarchimedean linear Hodge-Newton decomposition. The lemma shows that two equivalent matrices with coefficients in the ring of integers in an archimedean field have the same Newton polygon under suitable conditions. Secondly, we give an archimedean analogue of the above lemma.
\end{abstract}
\section{Introduction}
In matrix analysis, there is an important subject of metric properties of matrices over a complete field with a given norm. It plays a key role in p-adic differential equations, see for \cite{Ked}. In this paper, our purpose is to investigate the relationship between the norms of eigenvalues of a matrix and singular values of a matrix. We first present the main points in the archimedean setting and then for nonarchimedean field.

Let $A$ be an $n\times n$ matrix over $\mathbb{C}$ and $\mathbb{C}^n$ equipped with $L^2$ norm. The operator norm of $A$ is $|A|=\sup \limits_{v \in \mathbb{C}^n-\{0\}} \frac{|Av|}{|v|}$.  Throughout this paper, denote the eigenvalues of $A$ (resp. singular values) by $\lambda_{1},\dots,\lambda_{n}$ (resp. $\sigma_{1},\dots,\sigma_{n}$) and arrange so that $|\lambda_{1}| \geq |\lambda_{2}| \dots \geq |\lambda_{n}|$ (resp. $\sigma_{1} \geq \sigma_{2}\geq \dots \geq \sigma_{n}$).
There is an important decomposition of matrix called singular value decomposition.

\begin{theorem} (Singular value decomposition)
There exist unitary $n \times n$ matrices $U$ and $V$ such that $UAV=Diag(\sigma_1,\dots,\sigma_n)$.
\end{theorem}

As a direct consequence of singular value decomposition, the operator norm of $A$ equals the largest singular value $\sigma_{1}$. Using this consequence, Weyl gave the relationship between the norms of eigenvalues of a matrix and singular values of a matrix. In other words, there is an upper bound of norms of product of all eigenvalues.

\begin{theorem} \cite{Wey}
\begin{displaymath}
\sigma_{1}\dots \sigma_{i} \geq |\lambda_{1}\dots \lambda_{i}| \qquad (i=1,\dots,n),
\end{displaymath}
\vspace{0.2cm}
with the equality for $i=n$
\end{theorem}

The equality in Weyl's theorem has a structual meaning. 

\begin{theorem} \cite{Ked} (Archimedean linear Hodge-Newton decomposition) 

Suppose that for some $i \in \{1,\dots,n-1\}$ we have 

\[ \sigma_{i}>\sigma_{i+1}, \quad |\lambda_{i}|>|\lambda_{i+1}|, \]
\[ \sigma_{1}\dots \sigma_{i} = |\lambda_{1}\dots \lambda_{i}|. \]

Then there exists a unitary matrix $U$ such that $U^{-1}AU$ is block diagonal. The first block accounts for the first $i$ singular values and eigenvalues and the second accounts for the others.

\end{theorem}

All the above theorems have the nonarchimedean analogue. Here we will be concerned with the nonarchimedean linear Hodge-Newton decomposition. 

Let $F$ be a complete nonachimedean field and $A$ be an $n\times n$ matrix over $F$. Denote the ring of integers in $F$ by $\textbf{o}_F$.  $m_{F}$ is the maximal ideal of $\textbf{o}_F$.

\begin{theorem} \cite{Ked} (Nonachimedean linear Hodge-Newton decomposition)

Suppose that for some $i \in \{1,\dots,n-1\}$ we have 

\[ |\lambda_{i}|>|\lambda_{i+1}|, \quad \sigma_{1}\dots \sigma_{i} = |\lambda_{1}\dots \lambda_{i}|. \]

Then there exists $U \in GL_{n}(\textbf{o}_F)$ such that $U^{-1}AU$ is block upper triangular, with the top left block accounting for the first $i$ singular values and eigenvalues and the bottom right block accounting for the others.

Moreover, if $\sigma_{i}>\sigma_{i+1}$, then $U^{-1}AU$ is block diagonal.
\end{theorem}

Theorem 1.4 will be proved later followed by \cite{Ked}. It is more complicated than the archimedean case since we don't have the notion of orthogonality.

In the following, we give some applications in Newton and Hodge polygons of the matrix. Newton polygon and Hodge polygon are motivated by Katz conjecture. It asserts that the Newton polygon lies above the Hodge polygon and leads thus to a relation between the characteristic polynomial of Frobenius and the Hodge numbers of the original variety. See \cite{Ogu} for more details. The nonarchimedean Weyl inequality can be viewed as a linear analogue of Katz conjecture.

As an application of nonarchimedean linear Hodge-Newton decomposition, we have

\begin{theorem}
Under the hypotheses and notations of Theorem 1.4, if $U,V \in GL_{n}(O_F)$ are congruent to the identity matrix modulo $m_F$, then the Newton polygons of $A$ and $UAV$ are coincide.
\end{theorem}

Theorem 1.5 is a generaliztion of Lemma 5 of \cite{KF}. Lemma 5 deals with the case when $A \in GL_{n}(F)$ is diagonal and it is the key ingredient of the proof of a conjecture about modular forms in \cite{KF}. By a similar argument on singular values, one can also deduce that they have the same Hodge polygon.

Finally, we give an archimedean analogue of Lemma 5 of \cite{KF}. 

\begin{proposition}
Suppose $U,V \in GL_{n}(\mathbb{C})$ and $D$ is diagonal. If every principal minor of $VU$ equals 1, then the Newton polygons of $D$ and $UDV$ are coincide. 
\end{proposition}

\section{Wedge product of matrices}
The aim of the chapter is to show the description of the singular values and eigenvalues of wedge product of matrices. We prove this in archimedean case. For nonarchimedean case, the method also works.

Let $A$ be an $n\times n$ matrix over $\mathbb{C}$ and $\mathbb{C}^n$ equipped with $L^2$ norm.

\begin{definition}
The \textbf{singular values} of A, denoted by $\sigma_{1} \geq \sigma_{2}\geq \dots \geq \sigma_{n}$, is defined to be the square roots of the eigenvalues of $A^{*}A$, where $A^{*}$ is the conjugate transpose of $A$.
\end{definition}

\begin{theorem} (Singular value decomposition)
There exist unitary $n \times n$ matrices $U$ and $V$ such that $UAV=Diag(\sigma_1,\dots,\sigma_n)$.
\end{theorem}

Now we introduce the construction of wedge product of matrices that will be frequently used later.

\begin{definition}
Let $M$ be a module over a ring $R$. The i-th \textbf{wedge product} $\wedge^{i} M$ of $M$ is the $R$-module generated  
by the symbols $m_{1}\wedge \dots \wedge m_{i}$ for $m_{1},\dots,m_{i} \in M$, modulo the relations that the map $(m_{1},\dots,m_{i}) \mapsto m_{1}\wedge \dots \wedge m_{i}$ is $R$-linaer in each variable (fix the others) and alternating.
\end{definition}

\begin{remark}
(1)\, If $M$ is freely generated by $e_{1},\dots,e_{n}$, then $\{e_{j_{1}}\wedge \dots \wedge e_{j_{i}}: 1 \leq j_{1} <\dots <j_{i} \leq n \}$ form a basis of $\wedge^{i} M$.

(2)\,The wedge product is a functor on the category of $R$-modules. That is, any linear transformation $T:M \rightarrow N$ induces a linear transformation $\wedge^{i} T: \wedge^{i} M \rightarrow \wedge^{i} N$.
\end{remark}
The following proposition investigates the singular values and eigenvalues of wedge product of matrices.

\begin{proposition}
The singular values (resp. eigenvalues) of $\wedge^{i} A$ are the i-fold products of the singular values (resp. eigenvalues) of $A$.
\end{proposition}

\begin{proof}
The main idea of the proof is to take the appropriate decomposition of the matrix $A$. Fix a basis of $\mathbb{C}^{n}$, we get a basis of $\wedge^{i} A$ by taking i-fold exterior products of basis elements.

For the eigenvalues of $\wedge^{i} A$, according to the Jordan decomposition of $A$, it follows that there exsits $U\in GL_{n}(\mathbb{C})$ such that $U^{-1}AU$ is upper triangular with its eigenvalues on the diagonal. By acting on the basis, we have $(AB)\wedge (CD)= (A\wedge C)(B\wedge D)$, for any $A,B,C,D \in GL_{n}(\mathbb{C})$. Then
\[ \wedge^{i} (U^{-1}AU)= (\wedge^{i} U)^{-1} (\wedge^{i} A) (\wedge^{i} U). \]

Set $B=U^{-1}AU$. Denote the eigenvalues of $B$ and the corresponding eigenvectors by $\lambda_{1},\dots, \lambda_{n}$ and $v_{1},\dots,v_{n}$. Then

~$\quad (\wedge^{i} B)(v_{k_1} \wedge \dots \wedge v_{k_{i}})$~

~$=(Bv_{k_{1}})\wedge \dots (Bv_{k_{i}})$~

~$=(\lambda_{k_{1}}v_{k_{1}})\wedge \dots \wedge (\lambda_{k_{i}}v_{k_{i}})$~

~$=(\lambda_{k_{1}}\dots \lambda_{k_{i}})(v_{k_{1}}\wedge \dots \wedge v_{k_{i}}).$~

Since similar matrices have the same eigenvalues, we conclude that the eigenvalues of $\wedge^{i} A$ are the i-fold products of the eigenvalues of $A$.

Now consider the singular values of $\wedge^{i} A$. Let $\sigma_{1}, \dots, \sigma_{n}$ be the singular values of $A$. By applying Theorem 2.2, we can construct the unitary matrices $U, V$ such that  $UAV=Diag(\sigma_1,\dots,\sigma_n)$. Then
\[ \wedge^{i}(UAV)=(\wedge^{i} U)(\wedge^{i} A)(\wedge^{i} V). \]

Set $C=UAV=Diag(\sigma_1,\dots,\sigma_n)$ and $D=\wedge^{i} C$. Note that $C$ is a real matrix, we have $D^{*}D=D^{2}$. The singular values of $D$ is same as the eigenvlues of $D$. By the previous argument, it follows that the eigenvalues of $D$ are the i-fold products of singular values of $A$.
Finally, singular values of $\wedge^{i} A$ and $D=\wedge^{i}(UAV)$ are coincide, and consequently the singular values of $\wedge^{i} A$ are the i-fold products of singular values of $A$.
\end{proof}

\section{Nonarchimedean linear Hodge-Newton decomposition}
Let us first recall the definition of nonarchimedean field.

\begin{definition}
$(F,| \ |)$ is called \textbf{nonarchimedean field} if the norm satisfies the following

(a)\, $|x|=0$ if and only if $x=0$.

(b)\, $|xy|=|x||y|$.

(c)\, $|x+y|\leq max\{|x|,|y|\}$.
\end{definition}

Condition (c) is called strong triangle inequality and it takes a significant role in p-adic analysis. We call a nonarchimedean field is \textbf{complete} if every Cauchy sequence converges in $F$ with respect to the given norm.

\begin{assumption}
Throughout this section and the next, let $F$ be a complete nonarchimedean field and let $A=(a_{ij})$ be a $n\times n$ matrix over $F$.

View $A: F^{n}\rightarrow F^{n}$ as a linear transformation. The norm on $F^{n}$ is given by the supremum norm 

\[ |(x_{1},\dots,x_{n})|= max\{|x_{1}|,\dots,|x_{n}|\}. \]

and the corresonding operator norm is defined by

\[ |A|=\sup \limits_{v \in F^n-\{0\}} \frac{|Av|}{|v|}. \]

Consider the supremum norm on $F^n$ and $A$ acting on the standard basis of $F^n$, we have a simpler expression 

\[ |A|=max\{|a_{ij}|\}.\]
\end{assumption}

Now we introduce the notion of Newton polygon and Hodge polygon in order to give a geometric description of nonarchimedean linear Hodge-Newton polygon.

\begin{definition}
Given a sequence $s_{1},\dots,s_{n}$, one can define the \textbf{associated polygon} for this sequence to be the polygonal line joining the points

\[ (-n+i,s_{1}+\dots+s_{n})\qquad (i=0,\dots,n) \]
\end{definition}

\begin{definition}
Let $s_{1},\dots,s_{n}$ be the sequence with the property that $s_{1}+\dots+s_{i}$ is the minimum valuation of an $i\times i$ minor of $A$ for $i=1,\dots,n$. The associated polygon is called the \textbf{Hodge polygon} of $A$.

Define the \textbf{singular values} of $A$ by $\sigma_{1}=e^{-s_{1}},\dots,\sigma_{n}=e^{-s_{n}}$. Note 
that singular values are invariants under multiplication by a matrix in $GL_{n}(\textbf{o}_{F})$ and $\sigma_{1}=|A|$.
\end{definition}

\begin{definition}
Let $\lambda_{1},\dots \lambda_{n}$ be the eigenvalues of $A$ in some algebraic extension of $F$ with the unique norm extension. Arrange $|\lambda_{1}| \geq \dots \geq |\lambda_{n}|$. The associated polygon is called the \textbf{Newton polygon} of $A$.
\end{definition}

Before the nonarchimedean linear Hodge-Newton decomposition, let us introduce the definition of the generalized eigenspace.The motivation is that eigenspaces are not big enough to decompose a vector space and we need a good way to enlarge them.

\begin{definition}
Suppose $T$ is a linear transformation of vector space $V$. The \textbf{$\lambda$-generalized eigenspace} of $T$ is
\[ V_{[\lambda]}:=\{v\in V| (T-\lambda I)^{m} v=0. \quad for\, some\, m >0 \} \subset V.\]

Clearly 
\[V_{\lambda} \subset V_{[\lambda]}.\]
\end{definition}

\begin{theorem} (Theorem 8.21 \cite{Art})
Suppose $V$ is a finite dimensional vector space and $T$ is a linear transformation of vector space $V$. Denote all the distinct eigenvalues of $T$ by $\lambda_{1},\dots,\lambda_{r}$. Then the space $V$ is the direct sum of the generalized eigenspaces

\[ V=\oplus_{i=1}^{r} V_{[\lambda_i]}. \]
\end{theorem}
 
\begin{theorem} (Nonachimedean linear Hodge-Newton decomposition)

\vspace{0.1cm}
Suppose that for some $i \in \{1,\dots,n-1\}$ we have 

\[ |\lambda_{i}|>|\lambda_{i+1}|, \quad \sigma_{1}\dots \sigma_{i} = |\lambda_{1}\dots \lambda_{i}|. \]

\vspace{0.2cm}
Then there exists $U \in GL_{n}(\textbf{o}_F)$ such that $U^{-1}AU$ is block upper triangular, with the top left block accounting for the first $i$ singular values and eigenvalues and the bottom right block accounting for the others.

Moreover, if $\sigma_{i}>\sigma_{i+1}$, then $U^{-1}AU$ is block diagonal.
\end{theorem}
\begin{proof}
Here we follow the proof in \cite{Ked} and the construction of the matrix $U$ will be used later.

For the first claim, from Proposition 1.5.1 of \cite{Chr}, one can deduce that $P(T)=(T-\lambda_{1})\dots (T-\lambda_{i})$ and $Q(T)=(T-\lambda_{i+1})\dots (T-\lambda_{n})$ have coefficients in $F$. Since $P$ and $Q$ have no common roots, there exists $B,C \in F[T]$ such that $PB+QC=1$. Hence $P(A)B(A)$ and $Q(A)C(A)$ give the projectors for a direct sum decomposition separating the first $i$ generalized eigenspaces from the others.

In the language of basis, one can find a basis $v_{1},\dots,v_{n}$ of $F^n$ such that $v_{1},\dots,v_{i}$ span the generalized eigenspaces with eigenvalues $\lambda_{1},\dots,\lambda_{i}$ and $v_{i+1},\dots,v_{n}$ span the generalized eigenspaces with eigenvalues $\lambda_{i+1},\dots,\lambda_{n}$. Let $\omega_{1},\dots,\omega_{i}$ be s basis of $\textbf{o}_{F}^n \cap (Fv_{1}+\dots+Fv_{i})$ and then extends to the basis $\omega_{1},\dots,\omega_{n}$ of $\textbf{o}_{F}^n$. Let $e_{1},\dots,e_{n}$ be the standard basis of $F^n$ and define $U\in GL_{n}(\textbf{o}_{F})$ as $\omega_{j}=\sum_{i} U_{ij}e_{i}$. Then

\[ U^{-1}AU=\begin{pmatrix}
B & C\\
O & D\\
\end{pmatrix} \]

Writing
\[ U^{-1}AU=\begin{pmatrix}
B & O\\
O & D\\
\end{pmatrix}
\begin{pmatrix}
I_{i} & B^{-1}C \\
O & I_{n-i}\\
\end{pmatrix},\]

\vspace{0.2cm}
we deduce that the singular values of $B$ and $D$ together must comprise $\sigma_{1},\dots,\sigma_{n}$. Hence the product of the singular values of $B$ equals $\sigma_{1}\dots \sigma_{i}$.

Furthermore, $B^{-1}C$ must have entries in $\textbf{o}_{F}$. In fact, by Cramer's rule, each entry of $B^{-1}C$ is an $i \times i$ minor of $A$ divided by the determinant of $B$. Denote the singular values of $B$ by $\sigma^{\prime}_{1},\dots,\sigma^{\prime}_{i}$, then $\sigma_{1}\dots \sigma_{i}=\sigma^{\prime}_{1} \dots \sigma^{\prime}_{i}= \sqrt{det(B^{*}B)}=|det(B)|$. It follows that norm of entry of $B^{-1}C$ is no more than $|\lambda_{1}\dots \lambda_{i}|=\sigma_{1}\dots \sigma_{i}=|det(B)|$ and it leads to the fact as desired.

For the second claim, now assume that $\sigma_{i} >\sigma_{i+1}$. In this case, conjugating by the matrix

\[ \begin{pmatrix}
I_{i} & -B^{-1}C \\
O & I_{n-i} \\
\end{pmatrix} \]

Then get a new matrix 

\[ U^{-1}AU=\begin{pmatrix}
B & C_{1}\\
O & D\\
\end{pmatrix}. \]

\vspace{0.2cm}
where $C_{1}=B^{-1}CD$. Note that $|C_{1}| \leq |B^{-1}||C||D|= \sigma_{i}^{-1} |C| \sigma_{i+1} < |C|$ implies $|C_{k}| \leq (\frac{\sigma_{i+1}}{\sigma_{i}})^{k}|C| \rightarrow 0$ as $k \rightarrow \infty$. This process converges.
In this way, we obtain the sequence of matrices 

\[ U_k=\begin{pmatrix}
I_{i} & -B^{-1}C_{k-1} \\
O & I_{n-i} \\
\end{pmatrix}\in GL_{n}(\textbf{o}_{F}). \]

\vspace{0.2cm}
Set $C_{0}=C$. Then $|U_k-Id|=\frac{C_{k-1}}{\sigma_{i}}\rightarrow 0$ as $k \rightarrow \infty$ and so the convergent product $U=U_{1}U_{2}\dots$ satisfies 

\[ U^{-1}AU=\begin{pmatrix}
B & O\\
O & D\\
\end{pmatrix}, \]

as desired.
\end{proof}

\begin{remark}
The geometric meaning of the conditions of Theorem 3.6, says that the Newton polygon has a vertex with $x$-coordinate $-n+i$ and the vertex also lies on the Hodge polygon.
\end{remark}

\section{Applications}

Let us first present the main application of linear nonarchimedean Hodge-Newton decomposition.

\begin{theorem}
Under the hypotheses and notations of Theorem 3.8, if $U,V \in GL_{n}(\textbf{o}_F)$ are congruent to the identity matrix modulo $m_F$, then the Newton polygons of $A$ and $UAV$ are coincide.
\end{theorem}

One can start with the following lemma. It uses minors to characterize the coefficients of the characteristic polynomial.

\begin{lemma}
Denote $A=(a_{ij})_{n \times n}$. Then the characteristic polynomial of $A$ has the expression

~$\phi(\lambda)=det(\lambda I-A)$~

~$\, \qquad =\lambda^{n}-(a_{11}+\dots+a_{nn})\lambda^{n-1}+(\sum_{1 \leq i_{1} < i_{2} \leq n} A\begin{pmatrix}
i_{1} & i_{2} \\
i_{1} & i_{2} \\
\end{pmatrix}) \lambda^{n-2}+\dots+(-1)^k (\sum_{1\leq i_{1} < i_{2} <\dots<i_{k}\leq n}A \begin{pmatrix}
i_{1} & i_{2} & \dots & i_{k}\\
i_{1} & i_{2} & \dots & i_{k}\\
\end{pmatrix}) \lambda^{n-k} +\dots+(-1)^n det(A)
.$~

\begin{proof}
Consider the Taylor expansion of $\phi$ at 0 as follows. 
\[ \phi(\lambda)=\phi(0)+\phi^{'}(0) \lambda +\frac{\phi^{''}(0)}{2!} \lambda^2+\dots+\frac{\phi^{(k)}(0)}{k!}\lambda^k+\dots+\frac{\phi^{(n)}(0)}{n!} \lambda^n \]

\qquad $\phi(0)=
\begin{vmatrix}
\lambda-a_{11} & -a_{12} & \dots &-a_{1n} \\
-a_{21} & \lambda-a_{22} & \dots &-a_{2n} \\
\vdots & \vdots & & \vdots \\
-a_{n1} & -a_{n2} & \dots & \lambda-a_{nn} \\
\end{vmatrix}$

Then 
\[ \phi(0)=(-1)^{n} det(A). \]

\qquad $\phi^{'}(0)=\begin{vmatrix}
1 & -a_{12} & \dots &-a_{1n} \\
0 & \lambda-a_{22} & \dots &-a_{2n} \\
\vdots & \vdots & & \vdots \\
0 & -a_{n2} & \dots & \lambda-a_{nn} \\
\end{vmatrix}+\dots+\begin{vmatrix}
\lambda-a_{11} & -a_{12} & \dots &0 \\
-a_{21} & \lambda-a_{22} & \dots &0 \\
\vdots & \vdots & & \vdots \\
-a_{n1} & -a_{n2} & \dots & 1 \\
\end{vmatrix}$

implies 

\[ \phi^{'}(0)=(-1)^{n} \sum_{1 \leq i_{1} < \dots <i_{n-1} \leq n} A \begin{pmatrix}
 i_{1}& \dots & i_{n-1} \\
 i_{1}& \dots & i_{n-1} \\
 \end{pmatrix}
.\]

For higher order, $\phi^{(k)}(0)$ equals sum of $n(n-1)\dots (n-k+1)$ terms. Fix the order $1 \leq i_{1} < \dots <i_{k} \leq n$, $A \begin{pmatrix}
 i_{1}& \dots & i_{k} \\
 i_{1}& \dots & i_{k} \\
 \end{pmatrix}$ repeats $\frac{n(n-1)\dots(n-k+1)}{\binom{n}{k}}=k!$ times. 
 
Hence 
\[ \phi^{(k)}(0)=(-1)^{k} k! \sum_{1 \leq i_{1} < \dots <i_{k} \leq n} A \begin{pmatrix}
 i_{1}& \dots & i_{k} \\
 i_{1}& \dots & i_{k} \\
 \end{pmatrix}
.\]
\end{proof}

\end{lemma}

\begin{proposition}
Under the hypotheses and notations of Theorem 3.8, suppose that $U,V \in GL_{n}(\textbf{o}_{F})$ are congruent to the identity matrix modulo $m_{F}$. Then the product of the $i$ largest eigenvalues of $UAV$ again has norm $|\lambda_{1} \dots \lambda_{i}|$.
\end{proposition}

\begin{proof}
Denote the singular values of $A$ by $\sigma_{1},\dots,\sigma_{n}$ and suppose $\sigma_{i} >\sigma_{i+1}$. Applying nonarchimedean linear Hodge-Newton decomposition, there exists $U \in GL_{n}(\textbf{o}_{F})$ such that 
\[UAU^{-1}=\begin{pmatrix}
B & O \\
O & C \\
\end{pmatrix}\]
is block diagonal. 

Set $D=UAU^{-1}$ and $W=UV$. Write $UAV=(UAU^{-1})(UV)=D(UV)$. First let us consider the special case $i=1$. Since $U \equiv Id \,(mod\, m_{F}), V \equiv \, Id (mod \,m_{F})$, it follows that $W\equiv \,Id (mod \, m_{F}).$ By Binet-Cauchy formula, we have

\vspace{0.2cm}
$(DW) \begin{pmatrix}
i_{1} & i_{2} &\dots & i_{k}\\
i_{1} & i_{2} &\dots & i_{k}\\
\end{pmatrix}= \sum \limits_{1 \leq j_{1}< \dots < j_{k} \leq n} D \begin{pmatrix}
i_{1} & i_{2} &\dots & i_{k}\\
j_{1} & j_{2} &\dots & j_{k}\\
\end{pmatrix} W \begin{pmatrix}
j_{1} & j_{2} &\dots & j_{k}\\
i_{1} & i_{2} &\dots & i_{k}\\
\end{pmatrix}.$

\vspace{0.2cm}
Then

\vspace{0.2cm}
$|(DW)\begin{pmatrix}
i_{1} & i_{2} &\dots & i_{k}\\
i_{1} & i_{2} &\dots & i_{k}\\
\end{pmatrix}|= \max \limits_{1 \leq j_{1}< \dots < j_{k} \leq n}\{ |D \begin{pmatrix}
i_{1} & i_{2} &\dots & i_{k}\\
j_{1} & j_{2} &\dots & j_{k}\\
\end{pmatrix}| |W \begin{pmatrix}
j_{1} & j_{2} &\dots & j_{k}\\
i_{1} & i_{2} &\dots & i_{k}\\
\end{pmatrix}| \}.
$

\vspace{0.2cm}
Note that $W\equiv \,Id (mod \, m_{F}), |W \begin{pmatrix}
j_{1} & j_{2} &\dots & j_{k}\\
i_{1} & i_{2} &\dots & i_{k}\\
\end{pmatrix}| \leq 1$ by the strong triangle equality. It archives the maximum if and only $(j_{1},\dots,j_{k})=(i_{1},\dots,i_{k})$. For $D=UAU^{-1}=\begin{pmatrix}
B & O \\
O & C \\
\end{pmatrix}$, $B$ accounts for $\lambda_{1},\dots,\lambda_{i}$ and $D$ accounts for $\lambda_{i+1},\dots,\lambda_{n}$. Since the expansion of the determinant and $|\lambda_{1}| \geq \dots \geq |\lambda_{n}|$, we have

\vspace{0.2cm}
$
|(DW)\begin{pmatrix}
i_{1} & i_{2} &\dots & i_{k}\\
i_{1} & i_{2} &\dots & i_{k}\\
\end{pmatrix}|= |D\begin{pmatrix}
i_{1} & i_{2} &\dots & i_{k}\\
i_{1} & i_{2} &\dots & i_{k}\\
\end{pmatrix}|.$

\vspace{0.2cm}
By the expression of Lemma 4.2 and note that $|\lambda_{1}(UAV)+\dots+\lambda_{n}(UAV)|=|\lambda_{1}(UAV)|$, we obtain

\[ |\lambda_{1}(UAV)|=|\lambda_{1}(D)|=|\lambda_{1}| \]

For the general case, the idea is to reduce to the special case by considering the wedge product $\wedge^{i} A$. From Proposition 2.5, one can deduce that the largest eigenvalue of $\wedge^{i} A$ equals product of $i$ largest eigenvalues of $A$. It is clear that the norm of the largest eigenvalue of $\wedge^{i} A$ equals $|\lambda_{1} \dots \lambda_{i}|$. We proved the statement as desired.
\end{proof}

Now we can prove Theorem 4.1 as follows:

\vspace{0.1cm}
By Proposition 4.3,
\[|\lambda_{1}(UAV)\lambda_{2}(UAV) \dots \lambda_{i}(UAV)|=|\lambda_{1}(D)\lambda_{2}(D) \dots \lambda_{i}(D)|.\,\forall i\]

\vspace{0.1cm}
It implies that 
\[|\lambda_{k}(UAV)|=|\lambda_{k}(D)|\quad (k=1,\dots,n).\]

From the definition of Newton polygon, one obtains that the Newton polygons of $A$ and $UAV$ are coincide, which proved Theorem 4.1.

\vspace{0.1cm}
By a similar argument on singular values, we also show that $A$ and $UAV$ have the same Hodge polygon.

\begin{proposition}
Under the hypotheses and notations of Theorem 3.8, if $U,V \in GL_{n}(\textbf{o}_F)$ are congruent to the identity matrix modulo $m_F$, then the Hodge polygons of $A$ and $UAV$ are coincide.
\end{proposition}

\begin{proof}
The condition $U \equiv Id \,(mod\, m_{F})$ leads to $|U|=max\{|U_{ij}|\}=1$. We 
have $|V|=1$ by the same argument for $V$. Hence 
\[\sigma_{1}(UAV)=|UAV|=|U||A||V|=|A|=\sigma_{1}(A).\]

Consider the wedge product of $A$, then for any $k=1,\dots,n$
\[\sigma_{1}(A) \dots \sigma_{k}(A)=|\wedge^{k} A|=|(\wedge^{k} U)(\wedge^{k} A)(\wedge^{k} V)|=|\wedge^{k}(UAV)|=\sigma_{1}(UAV) \dots \sigma_{k}(UAV).\]

So the singular values of $A$ and $UAV$ are coincide. It immediately follows from the definition of Hodge polygon that $A$ and $UAV$ have the same Hodge polygon.
\end{proof}

Finally, we give an archimedean analogue of Lemma 5 of \cite{KF}.
\begin{proposition}
Suppose $U,V \in GL_{n}(\mathbb{C})$ and $D$ is diagonal. If every principal minor of $VU$ equals 1, then the Newton polygons of $D$ and $UDV$ are coincide. 
\end{proposition}

\begin{proof}
Here we give a direct proof. Note that $U^{-1}(UDV)U=D(VU)$, it is obvious that the Newton polygons of $UDV$ and $D(VU)$ are coincide.

Since $D$ is diagonal, we have 

\vspace{0.2cm}
~$|D(VU)\begin{pmatrix}
i_{1} & i_{2} &\dots & i_{r}\\
i_{1} & i_{2} &\dots & i_{r}\\
\end{pmatrix}|= |D\begin{pmatrix}
i_{1} & i_{2} &\dots & i_{r}\\
i_{1} & i_{2} &\dots & i_{r}\\
\end{pmatrix}||(VU)\begin{pmatrix}
i_{1} & i_{2} &\dots & i_{r}\\
i_{1} & i_{2} &\dots & i_{r}\\
\end{pmatrix}|$~

\vspace{0.1cm}
\qquad \qquad \qquad \qquad \qquad \qquad \qquad~$=|D\begin{pmatrix}
i_{1} & i_{2} &\dots & i_{r}\\
i_{1} & i_{2} &\dots & i_{r}\\
\end{pmatrix}|.$~

\vspace{0.1cm}
Using the same method in Proposition 4.3, one can obtain that the Newton polygons of $D$ and $D(UV)$ are coincide, which proved the statement.
\end{proof}

\vspace{0.2cm}

\end{document}